\documentclass[12pt]{amsart}

\usepackage{amsmath}
\usepackage{amssymb, amsfonts,amscd,verbatim,hyperref,cleveref,graphics}
\usepackage{xspace}
\usepackage{tikz}
\usetikzlibrary{matrix,calc,decorations.pathmorphing,shapes,arrows}

\newtheorem{theorem}{Theorem}[section]
\newtheorem{proposition}[theorem]{Proposition}
\newtheorem{corollary}[theorem]{Corollary}
\newtheorem{lemma}[theorem]{Lemma}
\newtheorem{definition}[theorem]{Definition}
\numberwithin{equation}{section}

\theoremstyle{definition}
\newtheorem{remark}[theorem]{Remark}
\newtheorem{example}[theorem]{Example}
\newtheorem{question}[theorem]{Question}

\DeclareSymbolFont{bbold}{U}{bbold}{m}{n}
\DeclareSymbolFontAlphabet{\mathbbold}{bbold}

\newcommand{\zs}

\newcommand{\term}[1]{{\textit{\textbf{#1}}}}   

\definecolor{dred}{RGB}{95,2,31}

\begin{document}

\title[Unbounded convergences]
{Completeness of Unbounded Convergences}

\author{M.A. Taylor}
\address{Department of Mathematical and Statistical Sciences,
         University of Alberta, Edmonton, AB, T6G\,2G1, Canada.}
\email{mataylor@ualberta.ca}


\thanks{The author acknowledges support from NSERC and the University of Alberta.}

\date{\today}

\begin{abstract}
As a generalization of almost everywhere convergence to vector lattices, unbounded order convergence has garnered much attention. The concept of boundedly $uo$-complete Banach lattices was introduced by N.~Gao and F.~Xanthos, and has been studied in recent papers by D.~Leung, V.G.~Troitsky, and the aforementioned authors. We will prove that a Banach lattice is boundedly $uo$-complete iff it is monotonically complete. Afterwards, we study completeness-type properties of minimal topologies; minimal topologies are exactly the Hausdorff locally solid topologies in which $uo$-convergence implies topological convergence.
\end{abstract}

\maketitle
\section{Introduction}
In the first half of the paper, we study when norm bounded $uo$-Cauchy nets in a Banach lattice are $uo$-convergent. The section starts with a counterexample to a question posed in \cite{LC}, and culminates in a proof that a Banach lattice is (sequentially) boundedly $uo$-complete iff it is (sequentially) monotonically complete. This gives the final solution to a problem that has been investigated in \cite{G}, \cite{GX}, \cite{GTX}, and \cite{GLX}.

The latter half of this paper focuses on the  ``extremal" topologies of a vector lattice $X$. For motivation, recall that corresponding to a dual pair $\langle E,E^*\rangle$ is a family of topologies on $E$  ``compatible" with duality. The two most important elements of this family are the weak and Mackey topologies, which are defined by their extremal nature. Analogously, given a vector lattice $X$, it is often possible to equip $X$ with many topologies compatible (in the sense of being locally solid and Hausdorff) with the lattice structure. It is easy to see that whenever $X$ admits some Hausdorff locally solid topology, the  collection of all Riesz pseudonorms on $X$ generates a \emph{finest} Hausdorff locally solid topology on $X$. This ``greatest" topology appears in many applications. Indeed, analogous to the theory of compatible locally convex topologies on a Banach space - where the norm topology is the Mackey topology - the norm topology on a Banach lattice $X$ is the finest topology on $X$ compatible with the lattice structure. This is \cite[Theorem 5.20]{AB03}. 

On the opposite end of the spectrum, a Hausdorff locally solid topology on a vector lattice $X$ is said to be \term{minimal} if there is no coarser Hausdorff locally solid topology on $X$; it is \term{least} if it is coarser than every Hausdorff locally solid topology on $X$. Least topologies were introduced in \cite{AB80} and studied in \cite{AB03}; minimal topologies were studied in \cite{LAB}, \cite{Conradie05}, \cite{me}, and \cite{KT}. An important example of a least topology is the unbounded norm topology on an order continuous Banach lattice. The unbounded absolute weak$^*$-topology on $L_{\infty}[0,1]$ is a noteworthy example of a minimal topology that is not least. In the next subsection, we briefly recall some facts about minimal and unbounded topologies; for a detailed exposition the reader is referred to \cite{me} and \cite{KT}.

\subsection{Notation}
Throughout this paper, all vector lattices are assumed Archimedean. For a net $(x_{\alpha})$ in a vector lattice $X$, we write $x_{\alpha}\xrightarrow{o}x$ if $(x_{\alpha})$ \textbf{\textit{converges to $x$ in order}}; that is, there is a net $(y_{\beta})$, possibly over a different index set, such that $y_{\beta}\downarrow 0$ and for every $\beta$ there exists $\alpha_0$ such that $|x_{\alpha}-x|\leq y_{\beta}$ whenever $\alpha\geq \alpha_0$. We write $x_{\alpha}\xrightarrow{uo}x$ and say that $(x_{\alpha})$ \textbf{\textit{uo-converges}} to $x\in X$ if $|x_{\alpha}-x|\wedge u\xrightarrow{o}0$ for every $u \in X_+$. For facts on $uo$-convergence, the reader is referred to \cite{GTX}. In particular, \cite[Theorem 3.2]{GTX} will be used freely. Recall that a Banach lattice $X$ is \term{(sequentially) boundedly uo-complete} if norm bounded $uo$-Cauchy nets (respectively, sequences) in $X$ are $uo$-convergent in $X$.

Given a locally solid topology $\tau$ on a vector lattice $X$, one can associate a topology, $u\tau$, in the following way. If $\{U_i\}_{i\in I}$ is a base at zero for $\tau$ consisting of solid sets, for each $i\in I$ and $u\in X_+$ define
$$U_{i,u}:=\{x\in X:\; |x|\wedge u\in U_i\}.$$
As was proven in \cite[Theorem 2.3]{me}, the collection $\mathcal N_0=\{U_{i,u}:\; i\in I, u\in X_+\}$ is a base of neighbourhoods at zero for a new locally solid topology, denoted by $u\tau$, and referred to as the \term{unbounded $\tau$-topology}. Noting that the map $\tau\mapsto u\tau$ from the set of locally solid topologies on $X$ to itself is idempotent, a locally solid topology $\tau$ is called \term{unbounded} if there is a locally solid topology $\sigma$ with $\tau=u\sigma$ or, equivalently, if $\tau=u\tau.$ The following connection between minimal topologies, unbounded topologies, and $uo$-convergence was proven in \cite[Theorem 6.4]{me}. Recall that a locally solid topology $\tau$ is \textbf{\textit{Lebesgue}} if order null nets are $\tau$-null.
\begin{theorem}\label{a.e. implies measure}
Let $\tau$ be a Hausdorff locally solid topology on a vector lattice $X$. TFAE:
\begin{enumerate}
\item\label{a.e. implies measure-1} $uo$-null nets are $\tau$-null;
\item $\tau$ is Lebesgue and unbounded;
\item $\tau$ is minimal.
\end{enumerate}
In particular, a vector lattice can admit at most one minimal topology.
\end{theorem}
Interestingly, the process of unbounding a topology can convert the greatest topology into the least topology; this happens with the norm topology on an order continuous Banach lattice.

All other undefined terminology is consistent with \cite{AB03}. In particular, we say that a locally solid topology $\tau$ on a vector lattice $X$ is \term{Levi} if $\tau$-bounded increasing nets in $X_+$ have supremum. Levi and \term{monotonically complete} are synonymous; the latter terminology is that of \cite{MN91}, and is used in \cite{GLX}.
\section{Boundedly uo-complete Banach lattices}
Results equating the class of boundedly $uo$-complete Banach lattices to the class of monotonically complete Banach lattices have been acquired, under technical assumptions, by N.~Gao, D.~Leung, V.G.~Troitsky, and F.~Xanthos. The sharpest result is \cite[Proposition 3.1]{GLX}; it states that a Banach lattice whose order continuous dual separates points is boundedly $uo$-complete iff it is monotonically complete. In this section, we remove the restriction on the order continuous dual.

The following question was posed as Problem 2.4 in \cite{LC}:
\begin{question}\label{Hui Li}
Let $(x_{\alpha})$ be a norm bounded positive increasing net in a Banach lattice $X$. Is $(x_{\alpha})$ $uo$-Cauchy in $X$?
\end{question}
If \Cref{Hui Li} is true, it is easily deduced that a Banach lattice is boundedly $uo$-complete iff it is monotonically complete. However, the next example answer this question in the negative, even for sequences.
\begin{example}
Let $S$ be the set of all non-empty finite sequences of natural numbers. For $s\in S$ define $\lambda(s)=\text{length}(s)$. If $s,t\in S$, define $s \leq t$ if $\lambda(s)\leq \lambda(t)$ and $s(i)=t(i)$ for $i=1,\dots, \lambda(s)$. For $s\in S$ with $\lambda(s)=n$ and $i\in \mathbb{N}$, define $s\ast i=(s(1),\dots,s(n),i)$. Put 
\begin{equation}
X=\{x\in \ell^{\infty}(S) : \lim_{i\to \infty} x(s\ast i)=\frac{1}{2}x(s)\ \text{for all}\ s\in S\}.
\end{equation}
It can be verified that $X$ is a closed sublattice of $(\ell^{\infty}(S), \|\cdot\|_{\infty})$ and for $t\in S$ the element $e^t: S \to \mathbb{R}$ defined by
\[
  e^t(s) =
  \begin{cases}
                                   (\frac{1}{2})^{\lambda(s)-\lambda(t)} & \text{if $t\leq s$} \\
                                   
  0  & \text{otherwise}
  \end{cases}
\]
is an element of $X$ with norm $1$. Define $f_1=e^{(1)}, f_2=e^{(1)}\vee e^{(2)}\vee e^{(1,1)}\vee e^{(1,2)}\vee e^{(2,1)}\vee e^{(2,2)}$, and, generally,
$$f_n=\sup\{e^t : \lambda(t)\leq n \ \text{and}\ t(k)\leq n \ \forall k\leq \lambda(t)\}.$$

The sequence $(f_n)$ is increasing and norm bounded by $1$; it was shown in \cite[Example 1.8]{BL88} that $(f_n)$ is not order bounded in $X^u$. Therefore, $(f_n)$ cannot be $uo$-Cauchy in $X$ for if it were then it would be $uo$-Cauchy in $X^u$ and hence order convergent in $X^u$ by \cite[Theorem 3.10]{GTX}. Since it is increasing, it would have supremum in $X^u$; this is a contradiction as $(f_n)$ is not order bounded in $X^u$. 
\end{example}
Under some mild assumptions, however, \Cref{Hui Li} has a positive solution. Recall that a Banach lattice is \term{weakly Fatou} if there exists $K\geq 1$ such that whenever $0\leq x_{\alpha}\uparrow x$, we have $\|x\|\leq K\sup \|x_{\alpha}\|$.
\begin{proposition}\label{weak fatou}
Let $X$ be a weakly Fatou Banach lattice. Then every positive increasing norm bounded net in $X$ is $uo$-Cauchy.
\end{proposition}
\begin{proof}
Let $K$ be such that $0\leq x_{\alpha}\uparrow x$ implies $\|x\|\leq K\sup\|x_{\alpha}\|$. Now assume that $0\leq u_{\alpha}\uparrow$ and $\|u_{\alpha}\|\leq 1$. Let $u>0$ and pick $n$ such that $\|u\|>\frac{K}{n}$. If $0\leq (\frac{1}{n}u_{\alpha})\wedge u\uparrow_{\alpha} u$, then $\|u\|\leq \frac{K}{n}$. Therefore, there exists $0<w \in X$ such that $(\frac{1}{n}u_{\alpha})\wedge u \leq u-w$ for all $\alpha$. But then $(nu-u_{\alpha})^+=n\left[u-(\frac{1}{n}u_{\alpha})\wedge u\right]\geq nw>0$ for all $\alpha$, so that $(u_{\alpha})$ is dominable. By \cite[Theorem 7.37]{AB03}, $(u_{\alpha})$ is order bounded in $X^u$, and hence $u_{\alpha}\uparrow \widehat{u}$ for some $\widehat{u}\in X^u$. This proves that $(u_{\alpha})$ is $uo$-Cauchy in $X^u$, hence in $X$.
\end{proof}
\begin{proposition}
Let $X$ be a weakly $\sigma$-Fatou Banach lattice. Then every positive increasing norm bounded sequence in $X$ is $uo$-Cauchy.
\end{proposition}
\begin{proof}
The proof is similar and, therefore, omitted.
\end{proof}
Even though \Cref{Hui Li} is false, the equivalence between boundedly $uo$-complete and Levi still stands. We show that now: 
\begin{theorem}\label{sequ boundedly uo}
Let $X$ be a Banach lattice. TFAE:
\begin{enumerate}
\item $X$ is $\sigma$-Levi;
\item $X$ is sequentially boundedly uo-complete;
\item Every increasing norm bounded $uo$-Cauchy sequence in $X_+$ has a supremum. 
\end{enumerate}
\end{theorem}
\begin{proof}
(i)$\Rightarrow$(ii): Let $(x_n)$ be a norm bounded $uo$-Cauchy sequence in $X$. WLOG, $(x_n)$ is positive; otherwise consider positive and negative parts. Define $e=\sum_{n=1}^{\infty}\frac{1}{2^n}\frac{x_n}{1+\|x_n\|}$ and consider $B_e$, the band generated by $e$. Then $(x_n)$ is still norm bounded and $uo$-Cauchy in $B_e$. Also, $B_e$ has the $\sigma$-Levi property for if $0\leq y_n\uparrow$ is a norm bounded sequence in $B_e$, then $y_n\uparrow y$ for some $y\in X$ as $X$ is $\sigma$-Levi. Since $B_e$ is a band, $y\in B_e$ and $y_n\uparrow y$ in $B_e$. We next show that there exists $u\in B_e$ such that $x_n\xrightarrow{uo}u$ in $B_e$, and hence in $X$. 

For each $m, n, n'\in \mathbb{N}$, since $|x_n\wedge me-x_{n'}\wedge me|\leq |x_n-x_{n'}|\wedge me$, the sequence $(x_n\wedge me)_n$ is order Cauchy, hence order converges to some $u_m$ in $B_e$ since the $\sigma$-Levi property implies $\sigma$-order completeness. The sequence $(u_m)$ is increasing and 
$$\|u_m\|\leq K\sup_n\|x_n\wedge me\|\leq K\sup_n\|x_n\|<\infty$$
where we use that $\sigma$-Levi implies weakly $\sigma$-Fatou. Since $B_e$ is $\sigma$-Levi, $(u_m)$ increases to an element $u\in B_e$. Fix $m$. For any $N,N'$ define $x_{N,N'}=\sup_{n\geq N, n'\geq N'}|x_n-x_{n'}|\wedge e$. Since $(x_n)$ is $uo$-Cauchy, $x_{N,N'}\downarrow 0$. Now, for each $m$, 
$$ |x_n\wedge me-x_{n'}\wedge me|\wedge e\leq |x_n-x_{n'}|\wedge e\leq x_{N,N'} \ \forall n \geq N, n'\geq N'.$$
Taking order limit in $n'$ yields
$$|x_n\wedge me- u_m|\wedge e\leq x_{N,N'}$$
Taking order limit in $m$ now yields:
$$|x_n-u|\wedge e \leq x_{N,N'}, \ \forall n\geq N,$$
from which it follows that $|x_n-u|\wedge e\xrightarrow{o}0$ in $B_e$. This yields $x_n\xrightarrow{uo} u$ in $B_e$ since $e$ is a weak unit of $B_e$.

The implication (ii)$\Rightarrow$(iii) is clear. For the last implication it suffices, by \cite[Theorem 2.4]{AW97}, to verify that every norm bounded laterally increasing sequence in $X_+$ has a supremum. Let $(x_n)$ be a norm bounded laterally increasing sequence in $X_+$. By \cite[Proposition 2.2]{AW97}, $(x_n)$ has supremum in $X^u$, hence is $uo$-Cauchy in $X^u$. It follows that $(x_n)$ is $uo$-Cauchy in $X$ and, therefore, by assumption, $uo$-converges to some $x\in X$. It is then clear that $x_n\uparrow x$ in $X$.
\end{proof}
\begin{theorem}
Let $X$ be a Banach lattice. TFAE:
\begin{enumerate}
\item $X$ is Levi;
\item $X$ is boundedly $uo$-complete;
\item Every increasing norm bounded $uo$-Cauchy net in $X_+$ has a supremum.
\end{enumerate}
\end{theorem}
\begin{proof}
If $X$ is Levi, then $X$ is boundedly $uo$-complete by \cite[Proposition 3.1]{GLX}. It is clear that (ii)$\Rightarrow$(iii) and the proof of (iii)$\Rightarrow$(i) is the same as in the last theorem but with \cite[Theorem 2.4]{AW97} replaced with \cite[Theorem 2.3]{AW97}.
\end{proof}

\section{Completeness of minimal topologies}
Throughout this section, $X$ is a vector lattice and $\tau$ denotes a locally solid topology on $X$. We begin with a brief discussion on relations between minimal topologies and the $B$-property. \Cref{Cauchy bounded} will be of  importance as many properties of locally solid topologies are stated in terms of positive increasing nets. For minimal topologies, these properties permit a uniform and efficient treatment.

The $B$-property was introduced as property (B,iii) by W.A.J.~Luxemburg and A.C.~Zaanen in \cite{LZ64}. It is briefly studied in \cite{AB03} and, in particular, it is shown that the Lebesgue property does not imply the $B$-property. We prove, however, that if $\tau$ is unbounded then this implication does indeed hold true:
\begin{definition}
A locally solid vector lattice $(X,\tau)$ satisfies the \term{B-property} if it follows from $0 \leq x_n \uparrow$ in $X$ and $(x_n)$ $\tau$-bounded that $(x_n)$ is $\tau$-Cauchy. An equivalent definition is obtained if sequences are replaced with nets.
\end{definition}
\begin{proposition}\label{B}
If $X$ is a vector lattice admitting a minimal topology $\tau$, then $\tau$ satisfies the B-property.
\end{proposition}
\begin{proof}
Suppose $\tau$ is minimal and $(x_n)$ is a $\tau$-bounded sequence satisfying $0 \leq x_n\uparrow$. By \cite[Theorem 7.50]{AB03}, $(x_n)$ is dominable. By \cite[Theorem 7.37]{AB03}, $(x_n)$ is order bounded in $X^u$ so that $x_n\xrightarrow{uo}u$ for some $u\in X^u$. In particular, $(x_n)$ is $uo$-Cauchy in $X^u$. It follows that $(x_n)$ is $uo$-Cauchy in $X$. Since $\tau$ is Lebesgue, $(x_n)$ is $u\tau$-Cauchy in $X$. Finally, since $\tau$ is unbounded, $(x_n)$ is $\tau$-Cauchy in $X$.
\end{proof}
\begin{proposition}\label{Cauchy bounded}
Let $X$ be a vector lattice admitting a minimal topology $\tau$, and $(x_{\alpha})$ an increasing net in $X_+$. TFAE:
\begin{enumerate}
\item $(x_{\alpha})$ is $\tau$-bounded;
\item $(x_{\alpha})$ is $\tau$-Cauchy.
\end{enumerate}
\end{proposition}
\begin{proof}
It remains to prove (ii)$\Rightarrow$(i): Let $(x_{\alpha})$ be an increasing $\tau$-Cauchy net in $X_+$. By \cite[Corollary 5.7]{me}, $\tau$ extends to a complete Hausdorff Lebesgue topology $\tau^u$ on $X^u$. It follows that $x_{\alpha}\xrightarrow{\tau^u}x$ for some $x\in X^u$. Since $(x_{\alpha})$ is increasing, $x_{\alpha}\uparrow x$ in $X^u$. In particular, $(x_{\alpha})$ is order bounded in $X^u$, hence dominable in $X$ by \cite[Theorem 7.37]{AB03}. By \cite[Theorem 5.2]{me}, $(x_{\alpha})$ is $\tau$-bounded in $X$.
\end{proof}

Recall the following definition, taken from  \cite[Definition 2.43]{AB03}.
\begin{definition}
A locally solid vector lattice $(X,\tau)$ is said to satisfy the \textbf{\textit{monotone completeness property (MCP)}} if every increasing $\tau$-Cauchy net of $X_+$ is $\tau$-convergent in $X$. The $\sigma$-MCP is defined analogously with nets replaced with sequences. 
\end{definition}
\begin{remark}
By \Cref{Cauchy bounded}, a minimal topology has MCP iff it is Levi.
\end{remark}
\begin{proposition}\label{MCP}
Let $\tau$ be a Hausdorff locally solid topology on $X$. If $u\tau$ satisfies MCP then so does $\tau$. If $u\tau$ satisfies $\sigma$-MCP then so does $\tau$.
\end{proposition}
\begin{proof}
Suppose $0\leq x_{\alpha}\uparrow$ is a $\tau$-Cauchy net. It is then $u\tau$-Cauchy and hence $u\tau$-converges to some $x\in X$. Therefore, $x_{\alpha}\uparrow x$ and $x_{\alpha}\xrightarrow{\tau}x$. Replacing nets with sequences yields the $\sigma$-analogue.
\end{proof}
Recall by \cite[Theorem 2.46 and Exercise 2.11]{AB03} that a Hausdorff locally solid vector lattice $(X,\tau)$ is (sequentially) complete iff order intervals are (sequentially) complete and $\tau$ has ($\sigma$)-MCP. Therefore, since $\tau$-convergence agrees with $u\tau$-convergence on order intervals, $u\tau$ being (sequentially) complete implies $\tau$ is (sequentially) complete.

\begin{lemma}\label{unbounded}
Let $(X,\tau)$ be a Hausdorff locally solid vector lattice. If $\tau$ is unbounded then TFAE: 
\begin{enumerate}
\item $\tau$ has MCP and is pre-Lebesgue;
\item $\tau$ is Lebesgue and Levi.
\end{enumerate}
\end{lemma}
\begin{proof}
It is sufficient, by \cite[Theorem 2.5]{c0}, to prove that $(X,\tau)$ contains no lattice copy of $c_0$. Suppose, towards contradiction, that $X$ does contain a lattice copy of $c_0$, i.e., there is a homeomorphic Riesz isomorphism from $c_0$ onto a sublattice of $X$. This leads to a contradiction as the standard unit vector basis is not null in $c_0$, but the copy in $X$ is by \cite[Theorem 4.2]{me}.
\end{proof}

\Cref{unbounded} is another way to prove that a minimal topology has MCP iff it is Levi. We next present the sequential analogue:
\begin{lemma}\label{unbounded sigma}
Let $(X,\tau)$ be a Hausdorff locally solid vector lattice. If $\tau$ is unbounded then TFAE: 
\begin{enumerate}
\item $\tau$ has $\sigma$-MCP and is pre-Lebesgue;
\item $\tau$ is $\sigma$-Lebesgue and $\sigma$-Levi.
\end{enumerate}
\end{lemma}
\begin{proof}
(i)$\Rightarrow$(ii) is similar to the last lemma; apply instead \cite[Proposition 2.1 and Theorem 2.4]{c0}.

(ii)$\Rightarrow$(i): It suffices to show that $\tau$ is pre-Lebesgue. For this, suppose that $0\leq x_n\uparrow \leq u$; we must show that $(x_n)$ is $\tau$-Cauchy. Since $\tau$ is $\sigma$-Levi and order bounded sets are $\tau$-bounded, $x_n\uparrow x$ for some $x\in X$. Since $\tau$ is $\sigma$-Lebesgue, $x_n\xrightarrow{\tau}x$. 
\end{proof}

Putting pieces together from other papers, we next characterize sequential completeness of $uo$-convergence. 
\begin{theorem}\label{uo sigma comp}
Let $X$ be a vector lattice. TFAE:
\begin{enumerate}
\item $X$ is sequentially uo-complete;
\item Every positive increasing uo-Cauchy sequence in $X$ uo-converges in $X$;
\item $X$ is universally $\sigma$-complete.
\end{enumerate}
In this case, uo-Cauchy sequences are order convergent.
\end{theorem}
\begin{proof}
(i)$\Rightarrow$(ii) is clear. (ii)$\Rightarrow$(iii) by careful inspection of \cite[Proposition 2.8]{LC}, (iii)$\Rightarrow$(i) and the moreover clause follow from \cite[Theorem 3.10]{GTX}.
\end{proof}
\begin{remark} Recall that by \cite[Theorem 7.49]{AB03}, every locally solid topology on a universally $\sigma$-complete vector lattice satisfies the pre-Lebesgue property. Using $uo$-convergence, we give a quick proof of this. Suppose $\tau$ is a locally solid topology on a universally $\sigma$-complete vector lattice $X$; we claim that $uo$-null sequences are $\tau$-null. This follows since $\tau$ is $\sigma$-Lebesgue and $uo$ and $o$-convergence agree for sequences by \cite[Theorem 3.9]{GTX}. In particular, since disjoint sequences are $uo$-null, disjoint sequences are $\tau$-null.
\end{remark}
We next give the topological analogue of \Cref{uo sigma comp}:
\begin{lemma}
Let $X$ be a vector lattice admitting a minimal topology $\tau$. TFAE:
\begin{enumerate}
\item $\tau$ is $\sigma$-Levi;
\item $\tau$ has $\sigma$-MCP;
\item $X$ is universally $\sigma$-complete.
\item $(X,\tau)$ is sequentially boundedly $uo$-complete in the sense that $\tau$-bounded $uo$-Cauchy sequences in $X$ are $uo$-convergent in $X$.
\end{enumerate}
\end{lemma}
\begin{proof}
(i)$\Leftrightarrow$(ii) follows from \Cref{unbounded sigma}. We next deduce (iii). Since $\tau$ is $\sigma$-Levi, $X$ is $\sigma$-order complete; we prove $X$ is laterally $\sigma$-complete. Let $\{a_n\}$ be a countable collection of mutually disjoint positive vectors in $X$, and define $x_n=\sum_{k=1}^n a_k$. Then $(x_n)$ is a positive increasing sequence in $X$, and it is $uo$-Cauchy, as an argument similar to \cite[Proposition 2.8]{LC} easily shows. By \Cref{a.e. implies measure}, $(x_n)$ is $\tau$-Cauchy, hence $x_n\xrightarrow{\tau}x$ for some $x\in X$ since $\tau$ has $\sigma$-MCP. Since $(x_n)$ is increasing and $\tau$ is Hausdorff, $x_n\uparrow x$. Clearly, $x=\sup\{a_n\}$.



(iii)$\Rightarrow$(iv) follows from \Cref{uo sigma comp}; (iv)$\Rightarrow$(i) is easy.
\end{proof}
The following question(s) remain open:
\begin{question} \label{Whole space}
Let $X$ be a vector lattice admitting a minimal topology $\tau$. Are the following equivalent?
\begin{enumerate}
\item $(X,\tau)$ is sequentially complete;
\item $X$ is universally $\sigma$-complete.
\end{enumerate}
\end{question}
\begin{question} \label{Intervals}
Let $(X,\tau)$ be Hausdorff and Lebesgue. Are the following equivalent?
\begin{enumerate}
\item Order intervals of $X$ are sequentially $\tau$-complete;
\item $X$ is $\sigma$-order complete.
\end{enumerate}
\end{question}
\begin{remark}
\Cref{Whole space} and \Cref{Intervals} are equivalent. Indeed, in both cases it is known that (i)$\Rightarrow$(ii). If \Cref{Intervals} is true then \Cref{Whole space} is true since we have already established that minimal topologies have $\sigma$-MCP when $X$ is universally $\sigma$-complete. Suppose \Cref{Whole space} is true. If $X$ is $\sigma$-order complete, then $X$ is an ideal in its universal $\sigma$-completion, $X^s$. Indeed, it is easy to establish that if $Y$ is a $\sigma$-order complete vector lattice sitting as a super order dense sublattice of a vector lattice $Z$, then $Y$ is an ideal of $Z$; simply  modify the arguments in \cite[Theorem 1.40]{AB03}. By \cite[Theorem 4.22]{AB03} we may assume, WLOG, that $\tau$ is minimal. $\tau$ then lifts to the universal completion and can be restricted to $X^s$. 
\end{remark}
\Cref{Intervals} is a special case of  Aliprantis and Burkinshaw's \cite[Open Problem 4.2]{AB78}, which we state as well:
\begin{question}\label{Open A&B}
Suppose $\tau$ is a Hausdorff $\sigma$-Fatou topology on a $\sigma$-order complete vector lattice $X$. Are the order intervals of $X$ sequentially $\tau$-complete?
\end{question}
The case of complete order intervals is much easier than the sequentially complete case. The next result is undoubtedly known, but fits in nicely; we provide a simple proof that utilizes minimal topologies.
\begin{proposition}
Suppose $\tau$ is a Hausdorff Lebesgue topology on $X$. Order intervals of $X$ are complete iff $X$ is order complete.
\end{proposition}
\begin{proof}
If $X$ is order complete then order intervals are complete by \cite[Theorem 4.28]{AB03}.

By \cite[Theorem 4.22]{AB03} we may assume, WLOG, that $\tau$ is minimal. If order intervals are complete then $X$ is an ideal of $\widehat{X}=X^u$ by \cite[Theorem 2.42]{AB03} and \cite[Theorem 5.2]{me}. Since $X^u$ is order complete, so is $X$. 
\end{proof}
\begin{remark} If $X$ is an order complete and laterally $\sigma$-complete vector lattice admitting a minimal topology $\tau$, then $\tau$ is sequentially complete. Although these conditions are strong, they do not force $X$ to be universally complete. This can be seen by equipping the vector lattice of \cite[Example 7.41]{AB03} with the minimal topology given by restriction of pointwise convergence from the universal completion.
\end{remark}

The key step in the proof of \Cref{sequ boundedly uo} is \cite[Theorem 2.4]{AW97} which states that a Banach lattice is $\sigma$-Levi if and only if it is \emph{laterally} $\sigma$-Levi. A sequence $(x_n)$ in a vector lattice  is said to be \term{laterally increasing} if it is increasing and $(x_m-x_n)\wedge x_n=0$ for all $m\geq n$. We say that a locally solid vector lattice $(X,\tau)$ has the \term{lateral $\sigma$-Levi property} if $\sup x_n$ exists whenever $(x_n)$ is laterally increasing and $\tau$-bounded. For minimal topologies, the $\sigma$-Levi and lateral $\sigma$-Levi properties do not agree, as we now show:
\begin{proposition}
Let $X$ be a vector lattice admitting a minimal topology $\tau$. TFAE:
\begin{enumerate}
\item $X$ is laterally $\sigma$-complete;
\item $\tau$ has the lateral $\sigma$-Levi property;
\item Every disjoint positive sequence, for which the set of all possible finite sums is $\tau$-bounded, must have a supremum.
\end{enumerate}
\end{proposition}
\begin{proof}
(i)$\Rightarrow$(iii) is clear, as is (ii)$\Leftrightarrow$(iii); we prove (ii)$\Rightarrow$(i). Assume (ii) and let $(x_n)$ be a disjoint sequence in $X_+$. Since $(x_n)$ is disjoint, $(x_n)$ has a supremum in $X^u$. Define $y_n=x_1\vee \cdots \vee x_n$. The sequence $(y_n)$ is laterally increasing and order bounded in $X^u$. By \cite[Theorem 7.37]{AB03}, $(y_n)$ forms a dominable set in $X_+$. By \cite[Theorem 5.2(iv)]{me}, $(y_n)$ is $\tau$-bounded, and hence has supremum in $X$ by assumption. This implies that $(x_n)$ has a supremum in $X$ and, therefore, $X$ is laterally $\sigma$-complete.
\end{proof}
In \cite{Lab84} and \cite{Lab85}, many completeness-type properties of locally solid topologies were introduced. For entirety, we classify the remaining properties, which he refers to as ``BOB" and ``POB".
\begin{definition}
A Hausdorff locally solid vector lattice $(X,\tau)$ is said to be \textbf{boundedly order-bounded (BOB)} if increasing $\tau$-bounded nets in $X_+$ are order bounded in $X$. $(X,\tau)$ satisfies the \textbf{pseudo-order boundedness property (POB)} if increasing $\tau$-Cauchy nets in $X_+$ are order bounded in $X$.
\end{definition}
\begin{remark}
It is clear that a Hausdorff locally solid vector lattice is Levi iff it is order complete and boundedly order-bounded.  It is also clear that BOB and POB coincide for minimal topologies.
\end{remark}
\begin{proposition}\label{BOB}
Let $X$ be a vector lattice admitting a minimal topology $\tau$. TFAE:
\begin{enumerate}
\item $(X,\tau)$ satisfies BOB;
\item $X$ is majorizing in $X^u$.
\end{enumerate}
\end{proposition}
\begin{proof}
(i)$\Rightarrow$(ii): Let $0\leq u\in X^u$. Since $X$ is order dense in $X^u$, there exists a net $(x_{\alpha})$ in $X$ such that $0\leq x_{\alpha}\uparrow u$. In particular, $(x_{\alpha})$ is order bounded in $X^u$, hence dominable in $X$ by \cite[Theorem 7.37]{AB03}. By \cite[Theorem 5.2]{me}, $(x_{\alpha})$ is $\tau$-bounded. By assumption, $(x_{\alpha})$ is order bounded in $X$, hence, $(x_{\alpha})\subseteq [0,x]$ for some $x\in X_+$. It follows that $u\leq x$, so that $X$ majorizes $X^u$.

(ii)$\Rightarrow$(i): Suppose $(x_{\alpha})$ is an increasing $\tau$-bounded net in $X_+$. It follows from \cite[Theorem 7.50]{AB03} that $(x_{\alpha})$ is dominable, hence order bounded in $X^u$. Since $X$ majorizes $X^u$, $(x_{\alpha})$ is order bounded in $X$.
\end{proof}
\begin{remark}
By \cite[Theorem 7.15]{AB03}, laterally complete vector lattices majorize their universal completions.
\end{remark}
\begin{remark}
If $\tau$ is a Hausdorff Fatou topology on $X$, it is easy to see that $(X,\tau)$ satisfies BOB iff every increasing $\tau$-bounded net in $X_+$ is order Cauchy in $X$. Compare with \Cref{Hui Li}.
\end{remark}
We next state the $\sigma$-analogue of \Cref{BOB}.
\begin{proposition}
Let $X$ be an almost $\sigma$-order complete vector lattice admitting a minimal topology $\tau$. TFAE:
\begin{enumerate}
\item $(X,\tau)$ satisfies $\sigma$-BOB;
\item $X$ is majorizing in the universal $\sigma$-completion $X^s$ of $X$.
\end{enumerate}
\end{proposition}
\begin{proof}
(i)$\Rightarrow$(ii) is similar to \Cref{BOB}.

(ii)$\Rightarrow$(i): Suppose $(x_n)$ is an increasing $\tau$-bounded sequence in $X_+$. It is then dominable in $X$, hence in $X^s$ by \cite[Lemma 7.11]{AB03}. It follows by \cite[Theorem 7.38]{AB03} that $(x_n)$ is order bounded in $X^s$. Since $X$ is majorizing in $X^s$, $(x_n)$ is order bounded in $X$.
\end{proof}
The next definition is standard in the theory of topological vector spaces:
\begin{definition}
Let $(E,\sigma)$ be a Hausdorff topological vector space. $E$ is \term{quasi-complete} if every $\sigma$-bounded $\sigma$-Cauchy net is $\sigma$-convergent.
\end{definition}
\begin{remark}Since Cauchy sequences are bounded, there is no sequential analogue of quasi-completeness.
\end{remark}

We finish with the full characterization of completeness of minimal topologies:
\begin{theorem}
Let $X$ be a vector lattice admitting a minimal topology $\tau$. TFAE:
\begin{enumerate}
\item $X$ is universally complete;
\item $\tau$ is complete;
\item $\tau$ satisfies MCP;
\item $\tau$ is Levi;
\item $\tau$ is quasi-complete;
\item $(X,\tau)$ is boundedly $uo$-complete in the sense that $\tau$-bounded $uo$-Cauchy nets in $X$ are $uo$-convergent in $X$.
\end{enumerate}
\end{theorem}
\begin{proof}
(i)$\Leftrightarrow$(ii) by \cite[Corollary 5.3]{me} combined with \cite[Theorem 6.4]{me}. Clearly, (ii)$\Rightarrow$(iii)$\Leftrightarrow$(iv). (iii)$\Rightarrow$(ii) since if $\tau$ satisfies MCP then $\tau$ is topologically complete by \cite[Corollary 4.39]{AB03}. We have thus established that (i)$\Leftrightarrow$(ii)$\Leftrightarrow$(iii)$\Leftrightarrow$(iv). It is clear that (ii)$\Rightarrow$(v), and (v)$\Rightarrow$(iii) by \Cref{Cauchy bounded}.

(ii)$\Rightarrow$(vi): Let $(x_{\alpha})$ be a $uo$-Cauchy net in $X$; $(x_{\alpha})$ is then $\tau$-Cauchy and hence $\tau$-convergent. The claim then follows from \cite[Remark 2.26]{me}.

(vi)$\Rightarrow$(iv): Suppose $0\leq x_{\alpha}\uparrow$ is $\tau$-bounded. $(x_{\alpha})$ is then $uo$-Cauchy, hence $uo$-convergent to some $x\in X$. Clearly, $x=\sup x_{\alpha}$.
\end{proof}

\begin{remark}
This is in good agreement with \Cref{MCP}. If the minimal topology satisfies MCP then \Cref{MCP} states that every Hausdorff Lebesgue topology satisfies MCP. Universally complete spaces, however, admit at most one Hausdorff Lebesgue topology by \cite[Theorem 7.53]{AB03}.
\end{remark}


\end{document}